\documentclass{article}
\usepackage{eurosym}
\usepackage{amssymb}
\usepackage{amsfonts}
\usepackage{amsmath}
\usepackage{mathrsfs}

\newtheorem{theorem}{Theorem}

\newtheorem{corollary}[theorem]{Corollary}

\newtheorem{definition}[theorem]{Definition}
\newtheorem{example}[theorem]{Example}

\newtheorem{lemma}[theorem]{Lemma}
\newtheorem{notation}[theorem]{Notation}

\newtheorem{proposition}[theorem]{Proposition}
\newtheorem{remark}[theorem]{Remark}

\newenvironment{proof}[1][Proof]{\noindent\textbf{#1.} }{\ \rule{0.5em}{0.5em}}

\begin{document}
\title{\textbf{Weak almost contact structure with B-metric  }}
\author{Cornelia-Livia Bejan$^{1}$, \c{S}emsi Eken Meri\c{c}$^{2}$}
\date{\vspace{-5ex}}
\maketitle


\bigskip

\textbf{Abstract} Corresponding to weak almost contact metric structures (defined recently by Rovenski and Wolak), we introduce and study here weak almost contact structures with B-metric, which generalize the classical almost contact structures with Norden metric (B-metric). Several geometric properties are obtained, some special classes are investigated and a lot of examples are constructed throughout this work.

\bigskip

\noindent\textbf{Mathematics Subject Classification (2020).} 53C15, 53C25, 53D15 \medskip

\noindent\textbf{Keywords.} Almost contact structures, Norden metric, Killing tensor field
\medskip

\section{Introduction}

The geometry of almost contact  metric manifolds, introduced in the 1960{'}s as odd dimensional counterparts of almost Hermitian manifolds, has undergone a great development (see \cite{Blair} and references therein). Since then, contact geometry has been applied in physics (particularly mechanics and thermodynamics), in the energetic and dynamic structure of systems, optics and ray geometry, in engineering (motion constraints) and so on. For a new approach of almost contact metric geometry, see \cite{Ilka}. Correspondingly, the geometry of almost contact manifolds with Norden metric (B-metric) (e. g. \cite{Ganchev1}, \cite{Ivanov}, \cite{Manev2}, \cite{Nakova}) stands for a natural extension to the odd dimensional case of the almost complex geometry with Norden metric (B-metric), (e. g. \cite{Ganchev}, \cite{Norden}, \cite{Oproiu}). The letter B was chosen by Norden to highlight that complex manifolds with isometric metrics (i.e. Hermitian) were studied in the West, but his theory, where the metric was anti-isometric with respect to the complex structure, is a contribution of the Eastern school (i.e. Vostok, which is written Boctok). \medskip																									

\noindent Recently, as a generalization of the almost contact structure, Rovenski and Wolak introduced in \cite{Rovenski1} a new notion, namely weak almost contact structure $(\varphi,Q,\xi,\eta)$. They call a Riemannian metric $g$ compatible with this structure, a weak almost contact metric structure $(\varphi,Q,\xi,\eta,g)$ (see \cite{Rovenski1} - \cite{Rovenski3}) and we stress that $g$ is Riemannian. \medskip

\noindent  To fill a gap in literature, our natural task here is to introduce and study the other case, namely weak almost contact structures $(\varphi,Q,\xi,\eta)$ admitting an anti-compatible metric $g$. We call this new notion a weak almost contact B-metric structure $(\varphi,Q,\xi,\eta,g)$, or a weak almost contact structure with B-metric. It generalizes almost contact structures with B-metric (Norden metric). Different from the above weak almost contact structure introduced by Rovenski and Wolak, the metric in our case is semi-Riemannian (not necessarily Riemannian) and we stress that the anti-compatibility of the metric in our case leads to different results. \medskip

\noindent We find here several geometric properties of weak almost contact B-metric structures, such as their splitting, their parallelism, their integrability, some Killing properties, and so on. We obtain the structural group acting on the underlying manifold. In the last section, we deal with some special classes of these structures. Throughout the paper, we construct a large number of illustrative examples of such manifolds, to support the motivation of this study.  \medskip

\section{Preliminaries}

We recall now an interesting concept, introduced recently in \cite{Rovenski1}, and then studied in \cite{Rovenski5} -  \cite{Rovenski3}:\medskip

\begin{definition} \label{d1}\cite{Rovenski1}
	A smooth odd-dimensional manifold $M^{2n+1} (n\geq 1)$, carries a \textit{weak almost contact structure} $(\varphi,Q,\xi,\eta),$ if it admits a $(1,1)$-tensor $\varphi$, a nonsingular $(1,1)-$tensor field $Q$, the characteristic vector field $\xi$ and its dual 1-form (contact form) $\eta$, satisfying,
	\begin{eqnarray}
		\varphi^2=-Q+\eta\otimes \xi, \ \label{eq1}
	\end{eqnarray}
	\begin{eqnarray}
		\eta(\xi)=1, \ \ \ \ \ Q\xi=\xi,   \label{eq1a}
	\end{eqnarray}
	such that $Ker\eta$ is $\varphi-$invariant, i.e., $\varphi(Ker\eta)\subset Ker\eta$. Then ($M^{2n+1}, \varphi,Q,\xi,\eta)$, is called a weak almost contact manifold.
\end{definition}

This structure $(\varphi,Q,\xi,\eta)$ generalizes the classical almost contact structure (see \cite {Blair} and the references therein), for which $Q=\text{Id}$ on $TM$.\medskip

 By $\eta(\xi)=1$, the contact form $\eta$ determines a smooth $2n-$dimensional distribution $\mathcal{D}:=Ker\eta$, called the contact distribution, which we assumed to be $\varphi-$invariant, i.e., 
\begin{eqnarray}
	\varphi X\in\mathcal{D}, \ \ \ X\in\mathcal{D},  \label{eq3}
\end{eqnarray}
as  in the classical theory of almost contact structure \cite{Blair}, where $Q=\text{Id}$ on $TM$.\medskip

\noindent The relations (\ref{eq1}), (\ref{eq3}) and the nondegeneracy of $Q$ yield the following:

\begin{proposition} \label{p1}
On a weak almost contact manifold $(M^{2n+1},\varphi,Q,\xi,\eta)$, one has:\medskip

\noindent (a) $TM=\mathcal{D}\oplus Span\{\xi\}$; \ \ \ (b) $Im\varphi=\mathcal{D}$; \ \ \ \ (c) $rank \varphi=2n$; \ \ \ \ (d) $\mathcal{D}$ is $Q-$invariant, i.e., $Q(\mathcal{D})\subset\mathcal{D}$;      \ \ \ (e) $\varphi\xi=0$;   \ \ \  (f) $\eta\circ\varphi=0$;  \ \ \ (g) $\eta\circ Q=\eta$. 
\end{proposition}

\noindent Some of these properties were shown in \cite{Rovenski1}. \medskip

\begin{corollary} \label{l1}
	\noindent Any vector field $Y\in\chi(M)$ can be written as 
	\begin{eqnarray}
		Y=\varphi W+\alpha\xi, \label{e1}
	\end{eqnarray}
	where $W\in\Gamma(\mathcal{D})$ and $\alpha$ is a smooth function on $M$. Moreover, the following statements are equivalent:\medskip
	
	\noindent (i) $Y\in\chi(M)$ is identically zero; (ii) $W$ and $\alpha$ vanish identically; (iii) Both $\varphi Y=0$ and $\eta(Y)=0$.

\end{corollary}
\begin{proof}
	From Proposition \ref{p1}, (a) and (b), we obtain (\ref{e1}). From Proposition \ref{p1} (f), we obtain $\alpha=0$ if and only if $\eta(Y)=0$. From Proposition \ref{p1} (e), the relation (\ref{eq1}) and the non-singularity of $Q$, we obtain $W=0$ if and only if $\varphi Y=0$. Since $Y=0$ if and only if both $W=0$ and $\alpha =0$, we complete the proof.	
\end{proof}\medskip

\noindent We notice that both Definition \ref{d1} and Corollary \ref{l1} don't use any metric, but we need them later on, when the B-metric will come into play.

\noindent The above notion of a weak almost contact structure represents the background  for the anti-compatible metric that we introduce in section 4, namely a weak almost contact B-metric.\medskip

\noindent To study some geometric properties of a weak almost contact B-metric structure, which is the main notion of our paper, we recall here the following two definitions:

\begin{definition}  \label{d5}
	Let $(\widetilde{N},g)$ be a (semi-)Riemannian manifold carrying a distribution $D$ and let $\nabla$ denote the Levi-Civita connection of $g$. We say that
	\medskip
	
	\noindent (i) $D$ is parallel, if $\nabla_XY\in\Gamma(D)$, for any $X\in\chi(\widetilde{N})$ and $Y\in \Gamma(D)$;\medskip
	
	\noindent (ii) $D$ is integrable if for every $x\in \widetilde{N}$ there exists an integral manifold $N$ of $D$ containing $x$ (i.e.,  $D_x=T_xN$);\medskip
	
	\noindent (iii)  a (1,1)-tensor field $\varphi$, a vector field $\xi$ and a 1-form $\eta$, is parallel if we have $\nabla\varphi=0$, $\nabla \xi=0$ and $\nabla\eta=0$, respectively; \medskip
	
	\noindent (iv) a vector field $\xi$ is geodesic, if $\nabla_{\xi}\xi=0$;\medskip
	
	\noindent (v) a (1,1)-tensor field $\varphi$ and a 1-form $\eta$ is parallel with respect to a vector field $\xi$, if we have $\nabla_{\xi}\varphi=0$ and $\nabla_{\xi}\eta=0$, respectively.
\end{definition}

\noindent The following notions are well known in Differential Geometry, Theoretical Pysics, Analysis, (see for instance \cite{Baird}). 

\begin{definition}  \label{dd5}
	Let $(N,g)$ be a (semi-)Riemannian manifold and let $\nabla$ be its Levi-Civita connection. We say that
	\medskip
	
	\noindent (i)  a (1,1)-tensor field $\phi$ is divergence-free if $\delta\phi=trace(\nabla\phi)=0$;\medskip
	
	\noindent (ii) a vector field $\xi$ is harmonic if $trace \ g(\nabla_{\bullet} \xi,\bullet)=0$;\medskip
	
	\noindent (iii) a 1-form $\eta$ is closed if $\delta\eta=trace(\nabla_{\bullet}\eta)\bullet=0$.
	
\end{definition}

\begin{remark}
 In Definition \ref{d1},  the notion given by (\ref{eq1}) contains only one vector field $\xi$. In \cite{Rov55}, Rovenski  defined  weak para-$f$-structure by (2), pp. 12, with $p$ vector fields $\xi_1, ..., \xi_p$. However, Definition \ref{d1} is not a particular case of weak para-$f$-structure studied by Rovenski, since if $Q$ is identity, then the restriction of $f$ to $Ker \eta$ in Rovenski's paper  is almost product $(f^2 = I)$, while in our paper $\varphi$ restricted to $Ker \eta$ is almost complex $(\varphi ^2 = - I)$.	
\end{remark}

\section{Normal Weak Almost Contact Manifolds}
\noindent We construct now a class of manifolds with weak almost contact structures.

\begin{example} \label{ex1}
	Let $M=N\times U$ be a product manifold, where $(N,J,h)$ is a K\"{a}hler manifold with a non-singular Ricci operator $\mathcal{S}$ and $U$ is 1-dimensional manifold whose tangent bundle is spanned by a vector field ${\bar{\xi}}$ with its dual form $\bar{\eta}$. We construct on $M$ a vector field $\xi=(0,\bar{\xi})$ and its dual 1-form $\eta=(0,\bar{\eta})$. Since any vector field $X$ on $M$, can be written as $X=(\bar{X},\alpha\bar{\xi})$, where $\bar{X}$ is tangent to $N$ and $\alpha$ is a function on $U$, we may define two (1,1)-tensor fields $Q$ and $\varphi$ on $M$, by $Q (X)=(\mathcal{S}^2\bar{X},\alpha\bar{\xi})$ and $\varphi (X)=(J\mathcal{S}\bar{X},0)$. Therefore $(M,\varphi,Q,\xi,\eta)$ turns out to be a weak almost contact manifold.	
\end{example} 

On a weak almost contact manifold $(M^{2n+1},\varphi,Q,\xi,\eta)$, the exterior derivative is denoted by $d\eta(X,Y)=X(\eta(Y))-Y(\eta(X))-\eta([X,Y]), \ \ X,Y\in\chi(M)$ and as in the classical case, \cite{Blair},  the Nijenhuis tensor field of $\varphi$ is given by $[\varphi,\varphi](X,Y)=\varphi^2[X,Y]+[\varphi X,\varphi Y]-\varphi[\varphi X,Y]-\varphi[X,\varphi Y]$, for any $ X,Y\in\chi(M)$. \medskip

Hence, it leads to the following tensor fields:
\begin{eqnarray}
		N^{(1)}(X,Y)&=&[\varphi,\varphi](X,Y)+d\eta(X,Y)\xi,  \label{n1}\\
		 N^{(2)}(X,Y)&=& (\mathcal{L}_{\varphi X}\eta)(Y)-(\mathcal{L}_{\varphi Y}\eta)(X)=d\eta(\varphi X,Y)-d\eta(\varphi Y,X), \label{n2}\\
		 N^{(3)}(X)&=& (\mathcal{L}_{\xi}\varphi)X=[\xi,\varphi X]-\varphi[\xi,X], \label{n3}\\
		 N^{(4)}(X)&=& (\mathcal{L}_{\xi}\eta)(X)=\xi(\eta(X))-\eta([\xi,X])=d\eta(\xi,X). \label{n4}
\end{eqnarray}

\begin{definition}
A weak almost contact structure $(\varphi,Q,\xi,\eta)$ on a manifold $M^{2n+1}$ is called normal if the tensor field $N^{(1)}$ is identically zero.	
\end{definition}

\begin{example}
	Let $(M^{2n+1},\varphi,Q,\xi,\eta)$ be a weak almost contact manifold, constructed as in Example \ref{ex1}. Since it's easy to see that $\eta$ is closed, it follows that $M$ is normal if and only if the Nijenhuis tensor field of $J\mathcal{S}$ vanish identically on $N$.
\end{example}

\begin{theorem} \label{t1}
	If a weak almost contact manifold $(M^{2n+1},\varphi,Q,\xi,\eta)$ is normal, then $N^{(3)}$ and $N^{(4)}$ are identically zero and 
	\begin{eqnarray}
		N^{(2)}(X,Y)=\eta([(Q-Id)X,\varphi Y]), \ \ \ \ \ X,Y\in\chi(M).       \label{n21}
	\end{eqnarray}
\end{theorem}

\begin{proof}
	Since  $(M^{2n+1},\varphi,Q,\xi,\eta)$ is normal, then $N^{(1)}(X,Y)=0$, for any $X,Y\in\chi(M)$ and by taking $Y=\xi$, then we obtain 
	$0=N^{(1)}(X,\xi)=[\varphi,\varphi](X,\xi)+d\eta(X,\xi)\xi=\varphi^2[X,\xi]+[\varphi X,\varphi\xi]-\varphi[\varphi X,\xi]+d\eta(X,\xi)\xi=-Q[X,\xi]+\eta([X,\xi])\xi	-\varphi[\varphi X,\xi]+d\eta(X,\xi)\xi.$ By applying $\eta$ to this relation and by using Proposition \ref{p1} (f) and (g), we obtain $d\eta(X,\xi)=0$, which shows that $N^{(4)}$ is identically zero. If we use $N^{(1)}=0$, $N^{(4)}=0$ and Proposition \ref{p1} (e), then $0=N^{(1)}(X,\xi)=[\varphi,\varphi](X,\xi)=\varphi^2[X,\xi]+[\varphi X,\varphi\xi]-\varphi[\varphi X,\xi]=\varphi(\varphi[X,\xi]-[\varphi X,\xi])=\varphi(N^{(3)}(X))$, which shows that $\varphi(N^{(3)}(X))=0$. On the other hand, by applying Proposition \ref{p1} (f)  and $N^{(4)}=0$, we obtain 
	\begin{eqnarray}
		\eta(N^{(3)}(X))=\eta([\xi,\varphi X])=-\eta(N^{(4)}(X))=0,	  \label{n5}
	\end{eqnarray}	
	 which from Corollary \ref{l1}, yields $N^{(3)}=0$.	 Since $N^{(1)}$ is identically zero, for any $X,Y\in\chi(M)$, we have $0=N^{(1)}(\varphi X,Y)=[\varphi,\varphi](\varphi X,Y)+d\eta(\varphi X,Y)\xi=\varphi^2[\varphi X,Y]+[\varphi^2 X,\varphi Y]-\varphi[\varphi^2 X,Y]-\varphi[\varphi X,\varphi Y]+d\eta(\varphi X,Y)\xi. $ By applying $\eta$, then from Proposition \ref{p1} (f) and (\ref{eq1a}), we obtain $0=\eta([\varphi^2 X,\varphi Y])+d\eta(\varphi X,Y)=-\eta([QX,\varphi Y])+\eta(X)\eta([\xi,\varphi Y])-(\varphi Y)\eta (X) 
	 +(\varphi X)\eta(Y)-\eta([\varphi X,Y])=-\eta([(Q-Id)X,\varphi Y])-\eta(X)N^{(4)}(Y)+N^{(2)}(X,Y)$, where we used the last but one equality  (\ref{n5}). Since $N^{(4)}$ is identically zero, we obtain (\ref{n21}), which complete the proof.	
\end{proof}\medskip

\begin{remark}
	We stress here that our Thereom \ref{t1} is different from Theorem 2.1 of \cite{Patra}, since the definition relations (\ref{eq1}), (\ref{eq1a}) used here are all different from the ones used in the above cited paper.	Also the conclusion of Thereom \ref{t1}  differs from Theorem 2.1 of \cite{Patra} and moreover, in our Thereom \ref{t1} we don't use the metric.
\end{remark}

\section{Weak Almost Contact B-Metric Manifolds}

We recall from \cite{Rovenski1} that a manifold $M^{2n+1}$ carries a weak almost contact metric structure $(\varphi,Q,\xi,\eta,g)$, if it is endowed with a weak almost contact structure $(\varphi,Q,\xi,\eta)$ and a Riemannian metric compatible with it, i.e.
\begin{eqnarray*}
	g(\varphi X,\varphi Y)=g(X,QY)-\eta(X)\eta(Y), \ \ \ X,Y\in\chi(M). 
\end{eqnarray*}
This notion generalizes the classical almost contact metric structures.

\medskip

\noindent We introduce here the main notion of our study as a dual counterpart of the above one: 

\begin{definition} \label{d22}
	Let $(M^{2n+1},\varphi,Q,\xi,\eta)$ be a weak almost contact manifold, such that there exists a (semi-)Riemannian metric $g$, which is anti-compatible with it, i.e., 
	\begin{eqnarray}
		g(\varphi X,\varphi Y)=-g(X,QY)+\eta(X)\eta(Y), \ \ \ X,Y\in\chi(M).  \label{eq4}
	\end{eqnarray}
	Then $(\varphi,Q,\xi,\eta,g)$ is called a weak almost contact B-metric structure, or a\textit{ weak almost contact structure with B-metric (Norden metric)}. In this case $(M^{2n+1},\varphi,Q,\xi,\eta,g)$ is said to be a weak almost contact B-metric manifold, or a \textit{weak almost contact manifold with B-metric (Norden metric)}.
\end{definition}

\begin{remark} (i) Between the notion we introduced by Definition \ref{d22} and the weak almost contact metric structure, given in \cite{Rovenski1}, there are many differences. For instance, the metric in our case can be semi-Riemannian. \medskip
		
\noindent (ii) Different from the case of an almost contact manifold with B-metric \cite{Ganchev1}, \cite{Manev2}, where the signature of the metric is (n+1,n) on the whole manifold and neutral (n,n) on its restriction to the contact distribution, in the general case of a weak almost contact B-metric manifold, these properties may or may not occur, i.e., the signature may or may not be as above, as in the following two examples:
\end{remark}

\begin{example}
On a three-dimensional torus $T^3$, let $\{X,Y,Z\}$ be a global frame of vector fields. With respect to it, we construct a weak almost contact B-metric structure $(\varphi,Q,\xi,\eta,g)$ by the following matrices: 
$$Q=
\begin{pmatrix}
	r & 0 & 0 \\
	0 & r  & 0 \\
	0 & 0 & 1
\end{pmatrix}, \ \ \
\varphi=
\begin{pmatrix}
	0 & -1 & 0 \\
	r & 0  & 0 \\
	0 & 0 & 0
\end{pmatrix}, \ \ \
g= \begin{pmatrix}
	-r & 0 & 0 \\
	0 & 1  & 0 \\
	0 & 0 & 1
\end{pmatrix},
$$
for any real number $r\in\mathbb{R}$, where we put $\xi=Z$ and we define $\eta$ such that $\eta(X)=\eta(Y)=0$, $\eta(Z)=1$. Obviously, the contact distribution is $\mathcal{D}=Span\{X,Y\}$. Hence, it turns out that:\medskip

\noindent (i) $(\varphi, Q, \xi,\eta,g)$ is a weak almost contact B-metric structure, for any $ r \in \mathbb{R}$;\smallskip

\noindent (ii)  the structure  $(\varphi,Q,\xi,\eta,g)$ is almost contact B-metric if and only if $r=1$.\smallskip

\noindent (iii) if $r<0$, then $g$ is Riemannian;\smallskip

\noindent (iv) if $r>0$, then $g$ is of signature $(2,1)$ and $g$ restricted to $\mathcal{D}$ is of neutral signature $(1,1)$. 	
\end{example}

\begin{example}
	Let $M^5$ be a 5-dimensional parallelizable manifold and let $\{X_1,...,X_5\}$ be a global frame of vector fields. With respect to it, we construct a weak almost contact B-metric structure $(\varphi,Q,\xi,\eta,g)$ by the following matrices: 
	$$Q=
	\begin{pmatrix}
		1 & 0 & 0 & 0 & 0 \\
		0 & 1  & 0 & 0 & 0\\
		0 & 0 & -1& 0 & 0\\
	    0 & 0 & 0 & -1 & 0\\
		0 & 0 & 0 & 0 & 1
	\end{pmatrix}, \ \ \
	\varphi=
	\begin{pmatrix}
		0 & 1 & 0 & 0 & 0 \\
-1 & 0  & 0 & 0 & 0\\
0& 0 & 1& 0 & 0\\
0 & 0 & 0 & 1 & 0\\
0 & 0 & 0 & 0 & 0
	\end{pmatrix}, \ \ \
	g= 	\begin{pmatrix}
		-1 & 0 & 0 & 0 & 0 \\
		0 & 1  & 0 & 0 & 0\\
		0 & 0 & 1& 0 & 0\\
		0 & 0 & 0 & 1 & 0\\
		0 & 0 & 0 & 0 & 1
	\end{pmatrix},
	$$
	 where we put $\xi=X_5$ and we define $\eta$ such that $\eta(X_1)=...=\eta(X_4)=0$, $\eta(X_5)=1$. Obviously, the contact distribution is $\mathcal{D}=Span\{X_1,...,X_4\}$. Hence, it turns out that:\medskip
	
	\noindent (i) $g$ is of signature $(4,1)$ and  \smallskip
	
	\noindent(ii) $g$ restricted to $\mathcal{D}$ is of signature (3,1), which is no longer of neutral signature.	
\end{example}

\begin{proposition} \label{p2} Let $(M^{2n+1},\varphi,Q,\xi,\eta,g)$ be a weak almost contact B-metric manifold with $\nabla$ its Levi-Civita connection. Then (i)	
		\begin{eqnarray}
		\eta &=& g(\cdot,\xi); \label{e5}
	\end{eqnarray}

\noindent (ii) $TM=\mathcal{D}\perp Span\{\xi\}$ and $g$, restricted to  both distributions is nondegenerate;\medskip

\noindent (iii)  $\nabla_X\xi$ is orthogonal to $\xi$, i.e., 
\begin{eqnarray}
	\nabla_X\xi\in\Gamma(\mathcal{D}), \ \ \ X\in\chi(M);  \label{m2}
\end{eqnarray}
	
\noindent	(iv) $\varphi$ and $Q$ are symmetric (with respect to $g$), i.e., 
	\begin{eqnarray}
		g(\varphi X,Y) = g(X,\varphi Y), \ \ \ 
			g(QX,Y) = g(X,Q Y), \ \ \ X,Y\in\chi(M); \label{eq101}
	\end{eqnarray}

\noindent (v) the (0,2)-tensor field $G$ defined as 
\begin{eqnarray}
	G(X,Y)&=&g(X,\varphi Y), \ \ \ X,Y\in\chi(M)    \label{e3}
\end{eqnarray}
is symmetric on $M$ and restricted to $\mathcal{D}$, is nondegenerate; \medskip

\noindent (vi) the (1,1)-tensor fields $\varphi$ and $Q$ commute, i.e.,
\begin{eqnarray}
	\varphi\circ Q &=& Q\circ\varphi. \label{eq9}
\end{eqnarray}
\end{proposition}
\begin{proof}
	Let $X\in\chi(M).$ The relations (\ref{eq4}), (\ref{eq1a}) and Proposition \ref{p1} (e) yield (i); \\

\noindent From (\ref{e5}) and (\ref{eq1a}), we obtain that the (semi-)Riemannian metric $g$ makes the splitting from the Proposition \ref{p1} (a) to be orthogonal. The non-degeneracy of $g$ on $M$ makes both restrictions $g\slash_{\mathcal{D}}$ and $g\slash_{Span\{\xi\}}$ nondegenerate. From (\ref{eq1a}) and (\ref{e5}), it follows (iii). \medskip

\noindent (iv) If we apply consecutively (\ref{e1}), Proposition \ref{p1} (f), the relations (\ref{eq4}), (\ref{eq1}) and Proposition \ref{p1} (e), we obtain $g(\varphi X,Y)=g(\varphi X,\varphi W+\alpha\xi)=-g(X,QW)+\eta(X)\eta(W)=g(X,\varphi^2 W-\eta(W)\xi)+\eta(X)\eta(W)=g(X,\varphi Y).$ Then the symmetry of $Q$ is obtained by interchanging $X$ and $Y$ in (\ref{eq4}).\medskip

\noindent (v) is obtained from (\ref{e3}), (iv) and Proposition \ref{p1} (b).   \medskip

\noindent (vi) By applying twice consecutively the relation (\ref{eq1}) and Proposition \ref{p1} (f), we obtain $Q\varphi X=-\varphi^3X+\eta\circ\varphi(X)\xi=-\varphi(-Q+\eta\circ\xi)X=\varphi QX$, for any $X\in\chi(M)$.
\end{proof}\medskip

\begin{remark}
(a)	Different from the case of a weak almost contact metric structure (see \cite{Rovenski1}), where $\varphi$ is skew-symmetric (i.e. anti-self-adjoint) and $G = g(·, \varphi·)$ gives a 2-form, in the case of a weak almost contact B-metric structure, $\varphi$  is symmetric (i.e. self-adjoint) and the (0,2)-tensor field $G$ is symmetric. Therefore, the notion of cosymplectic manifold, defined in the case of a weak almost contact metric structure, can not be defined here, in the case of a weak almost contact B-metric structure (since $G$ is no longer a two-form). However, we note that $Q$ is symmetric (i.e. self-adjoint) in both cases, namely weak almost contact metric structure and weak almost contact B-metric structure.\\

(b) On a weak almost contact B-metric manifold $(M^{2n+1},\varphi,Q,\xi,\eta)$, we can define $\tilde{g}=G+\eta\otimes \eta$, which based on Proposition \ref{p2} (v), turns out to be another B-metric (Norden metric) on $M$, of the same signature, as is $g$. As in the case of classical almost contact  structure with B-metric, we call $g$ and $\tilde{g}$ as twin B-metrics.
\end{remark}

\noindent Now, we provide a class of weak almost contact B-metric manifolds:

\begin{example}
Let $M=N\times U$ be a product manifold, where $(N,J,h)$ is a K\"{a}hler-Norden manifold with a non-singular Ricci operator $\mathcal{S}$ and $U$ is 1-dimensi-onal manifold whose tangent bundle is spanned by a vector field ${\bar{\xi}}$ with its dual form $\bar{\eta}$. As in Example \ref{ex1}, we construct on $M$ a vector field $\xi=(0,\bar{\xi})$, its dual 1-form $\eta=(0,\bar{\eta})$ and the (1,1)-tensor fields $Q$ and $\varphi$ on $M$, defined respectively by $Q (X)=(\mathcal{S}^2\bar{X},\alpha\bar{\xi})$ and $\varphi (X)=(J\mathcal{S}\bar{X},0)$, for any $X=(\bar{X},\alpha\bar{\xi})\in\chi(M)$,  where $\bar{X}\in\chi(N)$ and $\alpha$ is a function on $U$. Hence $(M,\varphi,Q,\xi,\eta)$ turns out to be a weak almost contact manifold. Moreover, for any $X=(\bar{X},\alpha\bar{\xi}),Y=(\bar{Y},\beta\bar{\xi})\in\chi(M)$, we define $g(X,Y)=h(\bar{X},\bar{Y})+\alpha\beta$ which turns out to be a semi-Riemannian metric on $M$. One has $g(\varphi X,\varphi Y)=g((J\mathcal{S}\bar{X},0),(J\mathcal{S}\bar{Y},0))=h(J\mathcal{S}\bar{X},J\mathcal{S}\bar{Y})=-h(\mathcal{S}\bar{X},\mathcal{S}\bar{Y})=-h(\bar{X},\mathcal{S}^2\bar{Y})$, where we have used that $h$ is a Norden metric and the Ricci operator $\mathcal{S}$ is symmetric. On the other hand, we have\smallskip

\noindent $-g(X,QY)+\eta(X)\eta(Y)=-g((\bar{X},\alpha\bar{\xi}),Q(\bar{Y},\beta\bar{\xi}))+\alpha\beta=-g((\bar{X},\alpha\bar{\xi}),(\mathcal{S}^2\bar{Y},\beta\bar{\xi}))\\
=-h(\bar{X},\mathcal{S}^2\bar{Y})$, which shows that the relation (\ref{eq4})
 holds good. Therefore \newline $(M,\varphi,Q,\xi,\eta,g)$ turns out to be a weak almost contact B-metric manifold.
\end{example}

\begin{proposition} \label{p3}
	Let $(M^{2n+1},\varphi,Q,\xi,\eta,g)$ be a normal weak almost contact B-metric manifold, whose contact form $\eta$ is closed. Then $N^{(2)}$, $N^{(3)}$ and $N^{(4)}$ are identically zero.
\end{proposition}

\begin{proof}
	The relations (\ref{n2}) and (\ref{n4}) yield $N^{(2)}=0$ and $N^{(4)}=0$. From (\ref{n3}), we have $N^{(3)}(X)= [\xi,\varphi X]-\varphi[\xi,X]$, for any $X\in\chi(M)$. From Proposition \ref{p1} (f) and (e), we obtain on one hand $\eta(N^{(3)}(X))=\eta([\xi,\varphi X])=-N^{(4)}(\varphi X)=0$ and on the other hand $\varphi N^{(3)}(X)=\varphi[\xi,\varphi X]-\varphi^2[\xi,,X]=\varphi^2[X,\xi]-\varphi[\varphi X,\xi]=[\varphi,\varphi](X,\xi)=N^{(1)}(X,\xi)-d\eta(X,\xi)\xi$, for any $X\in\chi(M)$, from (\ref{n1}). As $M$ is normal and $\eta$ is closed, by applying Corollary \ref{l1}, we complete the proof.	
\end{proof}\medskip


We use Definition \ref{d5}, Proposition \ref{p2} (i), (ii) and the relation (\ref{eq1a}), to obtain:

\begin{proposition} \label{p4}
	If $(M^{2n+1},\varphi,Q,\xi,\eta,g)$ is a weak almost contact B-metric manifold, then the following are equivalent: (i) $\mathcal{D}$ is parallel;  (ii) $\eta$ is parallel;  (iii) $\xi$ is parallel.\medskip
\end{proposition}

We need the above equivalences in the following:

\begin{proposition}
	Let $(M^{2n+1},\varphi,Q,\xi,\eta,g)$ be a weak almost contact B-metric manifold.

\noindent (a) $\mathcal{D}$ is integrable provided either $\varphi$ is parallel or one of the equivalent properties in Proposition \ref{p4} holds good;	\medskip

\noindent (b) If $\varphi$ is parallel, then $Q$ is parallel if and only if any of the equivalent properties in Proposition \ref{p4} is satisfied. \medskip

\noindent (c) If $\varphi$ is divergence free, then $Q$ is divergence free if and only if $\eta$ is co-closed and $\xi$ is geodesic.

\end{proposition}
\begin{proof}
(a) follows from Definition \ref{d5}, Proposition \ref{p2} (i), (ii), the relation (\ref{eq1a}) and the Frobenius theorem.

To show (b), if we assume that $\varphi$ is parallel, i.e., $\nabla_X\varphi Y= \varphi\nabla_XY$, for any $X,Y\in\chi(M)$, then we apply $\varphi$ and use (\ref{eq1}), to obtain 
\begin{eqnarray}
	(\nabla_XQ)Y&=&(\nabla_X\eta)(Y)\xi+\eta(Y)\nabla_X\xi, \ \ \ \text{for any} \ \ X,Y\in\chi(M). \label{m1}
\end{eqnarray}
If one of the equivalent properties in Proposition \ref{p4} holds good, then $Q$ is parallel. Conversely, if $Q$ is parallel, then the right hand side of (\ref{m1}) vanishes identically. If we apply $\varphi$ to the right hand side of (\ref{m1}), then from Proposition \ref{p1} (i) and Proposition \ref{p2} (ii), we obtain $\nabla_X\xi\in Span\{\xi\}$. From (\ref{m2}), it follows the parallelism of $\xi$ which shows that $(b)$ is satisfied. To prove (c), we assume that $\varphi$ is divergence-free, i.e., $\delta\varphi=0$, that is,  $trace(\nabla_{\bullet}\varphi \bullet)= trace(\varphi\nabla_{\bullet}\bullet)$. We apply $\varphi$ to this equality and then use (\ref{eq1}), to obtain a similar relation as (\ref{m1}), but in which the trace is taken in both sides, i.e.,
\begin{eqnarray}
	\delta Q&=& (\delta\eta ) \xi+\nabla_{\xi}\xi \label{m11}
\end{eqnarray}
From (\ref{m2}), $\delta Q$ is orthogonally decomposed by the relation (\ref{m11}), according to Proposition \ref{p2} (ii), which shows that $(c)$ is satisfied, and hence we complete the proof.	
\end{proof}

\begin{proposition}
Let $(M^{2n+1},\varphi,Q,\xi,\eta, g)$ be a weak almost contact B-metric manifold and let $\nabla$ denote the Levi-Civita connection of $g$. If 
\begin{eqnarray}
	d\eta(X,Y)=g(\nabla_X\xi,Y), \ \ \ X,Y\in\chi(M), \label{m3}
\end{eqnarray} 
	 then $\xi$ is Killing and $\mathcal{D}$ is parallel.
\end{proposition}

\begin{proof}
	We recall that $\xi$ is a Killing vector field (or infinitesimal isometry), if $\mathcal{L}_{\xi}g=0$, where $\mathcal{L}_{\xi}$ denotes the Lie derivative in the $\xi$ direction, which is defined by 
	\begin{eqnarray}
		(\mathcal{L}_{\xi}g)(X,Y)&:=&\xi(g(X,Y))-g([\xi,X],Y)-g(X,[\xi,Y])\label{m4}\\
		&=&g(\nabla_X\xi,Y)+g(\nabla_Y\xi,X), \ \ \ X,Y\in\chi(M).\nonumber
	\end{eqnarray}	
As the left hand side of the relation (\ref{m3}) is a 2-form, then so is its right hand side, which proves that in (\ref{m4}) the Lie derivative vanishes identically and therefore $\xi$ is Killing. From (\ref{m3}), we have 
\begin{eqnarray}
	X\eta(Y)-Y\eta(X)-\eta([X,Y])=g(\nabla_X\xi,Y), \ \ \ X,Y\in\chi(M). \label{m5}
\end{eqnarray}
Based on the decomposition from Proposition \ref{p2} (ii), we have the following cases:  \medskip

\noindent (a) If $X,Y\in\Gamma(\mathcal{D})$, then from (\ref{m5}), we obtain  $-\eta([X,Y])=X\eta(Y)-\eta(\nabla_XY)$, which shows that $\nabla_YX\in\Gamma(\mathcal{D})$. \medskip

\noindent (b) If $X=\xi$ and $Y\in\Gamma(\mathcal{D})$ or $Y=\xi$, then from (\ref{m2}), it follows that (\ref{m3}) is trivially satisfied.\medskip

\noindent (c) If $X\in\Gamma(\mathcal{D})$ and $Y=\xi$, then from (\ref{m2}), it follows that (\ref{m3}) yields $\nabla_{\xi}X\in\Gamma(\mathcal{D})$. Hence we complete the proof.
\end{proof}\medskip

By taking the derivative with respect to $\xi$ in the relation (\ref{eq4}), we obtain the following:

\begin{lemma}
If $(M^{2n+1},\varphi,Q,\xi,\eta, g)$ is a weak almost contact B-metric manifold, then	
\begin{eqnarray}
	g((\nabla_{\xi}\varphi)X,\varphi Y)+g((\nabla_{\xi}\varphi)Y,\varphi X)&=&-g(X,(\nabla_{\xi}Q)Y)+\eta(X)(\nabla_{\xi}\eta)(Y) \nonumber\\
	&+&\eta(Y)(\nabla_{\xi}\eta)(X), \ X,Y\in\chi(M). \label{w9}
\end{eqnarray}
\end{lemma}
From (\ref{w9}), one can easily obtain:
\begin{proposition}
Let $(M^{2n+1},\varphi,Q,\xi,\eta, g)$ be a weak almost contact B-metric manifold. If $\varphi$ and $\eta$ are parallel with respect to $\xi$, then $Q$ is parallel with respect to $\xi$.	
\end{proposition}
	 
\section{Special Classes of Weak Almost Contact B-metric Manifolds}

On a weak almost contact B-metric manifold $(M^{2n+1},\varphi,Q,\xi,\eta, g)$, let $\nabla$ be the Levi-Civita connection of the metric $g$ and let $F$ be the (0,3)-tensor field defined by 
\begin{eqnarray}
	F(X,Y,Z)&=& g((\nabla_X\varphi)Y,Z), \ \ \ X,Y,Z\in\chi(M). \nonumber 
\end{eqnarray}
From (\ref{eq1}), Proposition \ref{p1} (e), (f) and the symmetry of $\varphi$, we obtain the following:

\begin{eqnarray}
	F(X,Y,Z)&=&F(X,Z,Y), \ \ \ F(X,\xi,\xi)=0,\label{f1}\\
	F(X,\varphi Y,\varphi Z)&=&F(X,Y,Z)-\eta(Y)F(X,\xi,Z)-\eta(Z)F(X,Y,\xi)\label{f2}\\
&&	-g(\nabla_X(Q-I)Y,\varphi Z)+g(\nabla_X\varphi Y,(Q-I)Z), \ X,Y,Z\in\chi(M). \nonumber
\end{eqnarray}

\begin{notation}
	The following 1-forms are obtained from $F$:
	\begin{eqnarray}
		\theta(X)=g(\delta\varphi,X); \ \theta^{*}(X)=traceF(\cdot,\varphi\cdot,X); \ w(X)=F(\xi,\xi,X), \ X\in\chi(M).
	\end{eqnarray}
\end{notation}
\begin{proposition} \label{p6}
Let $(M^{2n+1},\varphi,Q,\xi,\eta, g)$ be a weak almost contact B-metric manifold, where $\varphi\slash_{\mathcal{D}}=rJ$, for a real number $r$ and the complex operator $J$. Then, the structural group of the manifold is $\mathcal{G}\times I$, where $$\mathcal{G}=\{\begin{pmatrix}
	A & B  \\
	-B & A   
\end{pmatrix}\in O(n,n) \slash \ A,B\text{ are matrices of type } n\times n  \}.$$	
\end{proposition}

\begin{remark}
 Under the hypothesis of Proposition \ref{p6}, the relation (\ref{f1}) remains unchanged, but the relation (\ref{f2}) simplifies to 
 \begin{eqnarray}
 	F(X,\varphi Y,\varphi Z)&=&r^2F(X,Y,Z)-\eta(Y)F(X,\xi,Z)\label{f3}\\
 	&&	-\eta(Z)F(X,Y,\xi), \ \ \ X,Y,Z\in\chi(M). \nonumber
 \end{eqnarray}	
\end{remark}

\begin{example}
	Let $\mathbb{R}^4_{2}$ be the 4-dimensional pseudo-Euclidean space of index 2, endowed with the inner product $<x,y>=-x_1y_1-x_2y_2+x_3y_3+x_4y_4$. On $\mathbb{R}^4$, seen as the real image of the complex 2-dimensional space, we denote by $J$ the canonical complex structure defined as $J(a\partial_1+b\partial_2+c\partial_3+d\partial_4)=-c\partial_1-d\partial_2+a\partial_3+b\partial_4$. Let $S(r)$ be the time-like sphere of the radius r:
	\begin{eqnarray*}
		S(r)=\{x\in\mathbb{R}^4_2, \ <x,x>=-r^2\},  \ r>0.
	\end{eqnarray*}
	Hence, the position vector $P=x_1\partial_1+...+x_4\partial_4$ is normal to the sphere $S(r)$ at each point of it. It turns out that $\{T,JT,\xi\}$ is a global frame of vector fields on $S(r)$, (as being orthogonal to $P$), where 
	\begin{eqnarray*}
		T&=&x_2\partial_1-x_1\partial_2+x_4\partial_3-x_3\partial_4,\\
		\xi&=&\frac{1}{r}(-x_3\partial_1-x_4\partial_2+x_1\partial_3+x_2\partial_4),
	\end{eqnarray*}
such that $<\xi,\xi>=1$. If we define $\varphi(U)=r(aJT-bT)$, for any $U=aT+bJT+c\xi\in\chi(S(r))$ and the metric $g$ to be the restriction of $<,>$ to the sphere $S(r)$, then we obtain a weak almost contact B-metric manifold $(S(r),\varphi,Q,\xi,\eta, g)$, where $\eta$ is the dual form of $\xi$ and $Q=-\varphi^2+\eta\otimes\xi$. 
\end{example}

We note that this example can be easily generalized in higher dimension, to the time-like hypersphere $S(r)$ of radius $r$, in the $(2n+2)-$dimensional pseudo-Euclidean space of index $n+1$.

\begin{theorem}
The decomposition
\begin{eqnarray}
	\mathcal{F}&=&	\mathcal{F}_1\oplus...\oplus	\mathcal{F}_{11} \nonumber
\end{eqnarray}	
is invariant under the action of $\mathcal{G}\times I$, where $	\mathcal{F}_i$ $(i=1,...,11)$ denote the same classes as the ones obtain in the classification of \cite{Ganchev1}.	
\end{theorem}
The class $\mathcal{F}_0$ of weak almost contact B-metric manifolds is defined by the condition $	F(X,Y,Z)= 0, \ \ \ X,Y,Z\in\chi(M). $\medskip

We provide an example of a manifold in the class $\mathcal{F}_0$.

\begin{example}
Let $\mathbb{R}^{2n+1}=\{(u^1,...,u^n;v^1,...,v^n;t)\slash u^{i},v^{i},t\in\mathbb{R}\}$. We define the structure $(\varphi,Q,\xi,\eta, g)$ on $\mathbb{R}^{2n+1}$ in the following way, for any real number $r\in\mathbb{R}$:
\begin{eqnarray*}
&&	\xi=\frac{\partial}{\partial t}, \ \ \ \ \ \eta=dt;\\
&&	\varphi(\frac{\partial}{\partial u^{i}})=r\frac{\partial}{\partial v^{i}}, \ \ \ \ \	\varphi(\frac{\partial}{\partial v^{i}})=-\frac{\partial}{\partial u^{i}}, \ \ \ \ \ \varphi(\frac{\partial}{\partial t})=0, \ \ \ \ \  \\ 
&&	Q(\frac{\partial}{\partial u^{i}})=r\frac{\partial}{\partial u^{i}}, \ \ \ \ \ Q(\frac{\partial}{\partial v^{i}})=r\frac{\partial}{\partial v^{i}}, \ \ \ \ \ Q(\frac{\partial}{\partial t})=\frac{\partial}{\partial t}, \ \ \ \ \ i=\overline{1,n}.\\
&& g(X,Y)=-r\delta_{ij}X^{i}Y^{j}+\delta_{ij}X^{n+i}Y^{n+j}+X^{2n+1}Y^{2n+1},
\end{eqnarray*}
where $X=X^{i}\frac{\partial}{\partial u_i}+X^{n+i}\frac{\partial}{\partial v^{i}}+X^{2n+1}\frac{\partial}{\partial t}$, \ $Y=Y^{i}\frac{\partial}{\partial u_i}+Y^{n+i}\frac{\partial}{\partial v^{i}}+Y^{2n+1}\frac{\partial}{\partial t}$ and $\delta_{ij}$ are the Kronecker's symbols.\medskip

\noindent If $\nabla$ is the Levi-Civita connection of the metric $g$, then one can see that $\varphi$ is parallel with respect to it. Hence $(\mathbb{R}^{2n+1},\varphi,Q,\xi,\eta, g)$ is a weak almost contact B-metric manifold in the class $\mathcal{F}_0$.	
\end{example}

\section{Weak Nearly Sasaki-Norden Manifolds}

In this section, we define and study a class of a weak almost contact B-metric manifolds.

\begin{definition}
A weak almost contact B-metric manifold $(M^{2n+1},\varphi,Q,\xi,\eta, g)$ is called weak nearly Sasakian B-metric manifold or weak nearly Sasaki-Norden manifold if
\begin{eqnarray}
	(\nabla_X\varphi)X&=&-g(X,X)\xi+\eta(X)X, \ \ \ X\in\chi(M).  \label{w1}
\end{eqnarray}
\end{definition}

\begin{remark}
Equivalent to (\ref{w1}), one has 	
\begin{eqnarray}
	(\nabla_X\varphi)Y+(\nabla_Y\varphi)X=-2g(X,Y)\xi+\eta(X)Y+\eta(Y)X, \ \ \ X,Y\in\chi(M).  \label{w2}
\end{eqnarray}	
\end{remark}

\begin{theorem} \label{t3}
	Let $(M^{2n+1},\varphi,Q,\xi,\eta, g)$ be a weak nearly Sasaki-Norden manifold. Then \medskip
	
	\noindent (i) $\xi$ is geodesic; (ii) $\eta$ is parallel with respect to $\xi$;  (iii) $d\eta(\xi,\cdot)=0$; (iv) $\varphi$ is never Killing; (v) the Lie-derivative satisfies
\begin{eqnarray}
	(\mathcal{L}_{\xi}g)(X,Y)=2g(X,\varphi Y), \ \ \ X,Y\in\chi(M),   \label{w5}
\end{eqnarray} 
provided 
\begin{eqnarray}
	(\nabla_XQ)Y=0, \ \  X\in\chi(M), \ Y\in\Gamma(\mathcal{D}). \label{w7}
\end{eqnarray} 
\end{theorem}

\begin{remark}
	In the above theorem, the statement (ii) can be written as  $\nabla_{\xi}\eta=0$ and the statement (iv) says that  the next relation is not true
	\begin{eqnarray}
		(\nabla_X\varphi)Y+(\nabla_Y\varphi)X=0, \ \ \ X,Y\in\chi(M). \label{w4}
	\end{eqnarray}
\end{remark}

\begin{proof}
(i)	We apply $\varphi$ to (\ref{w1}) and we use (\ref{eq1}), to obtain 
	\begin{eqnarray}
		\varphi\nabla_Y\varphi Y+Q\nabla_YY-\eta(\nabla_YY)\xi=\eta(Y)\varphi Y, \ \ \ Y\in\chi(M).     \nonumber
		\end{eqnarray}
	If we take $Y=\xi$, then from (\ref{m2}) it follows $Q\nabla_{\xi}\xi=0$. Hence, the nondegenerancy of $Q$ yields $\nabla_{\xi}\xi=0$. (ii) and (iii) follow from direct calculation using the fact that $\xi$ is geodesic. (iv) If we assume that (\ref{w4}) is satisfied, then from (\ref{w2}), it follows that the metric $g$ restricted to the distribution $\mathcal{D}$ is identical to zero, which contradicts the fact that $g$ is nondegenerate on the whole manifold $M$ (as we assumed $g$ is (semi-)Riemannian). We note that the relation (\ref{w2}) implies 
	\begin{eqnarray}
		(\nabla_{\xi}\varphi)X+(\nabla_X\varphi)\xi=\eta(X)\xi-X, \ \  X\in\chi(M), \label{w8}
	\end{eqnarray}
	which from Proposition \ref{p1} (e), the relation (\ref{eq4}) and (\ref{eq101}) yields
	\begin{eqnarray}
		g((\nabla_{\xi}\varphi)X,\varphi Y)=-g(\nabla_X\xi,QY)+ \eta(\nabla_X\xi)\eta(Y)-g(X,\varphi Y), \ X,Y\in\chi(M). \label{w6}
	\end{eqnarray}	
	
	Based on Proposition \ref{p2} (ii),  to obtain (v), it is enough to prove the following cases:\smallskip
	
	\noindent Case 1: If  $X=Y=\xi$, then (\ref{w5}) is satisfied, based on the relation (\ref{m2}).\smallskip
	
	\noindent Case 2: If $X\in\Gamma(\mathcal{D})$ and $Y=\xi$ (or $X=\xi$ and $Y\in\Gamma(\mathcal{D})$), then (\ref{w5}) follows from (\ref{m2}) and the parallelism of $\xi$. \smallskip
	
	\noindent Case 3: If $X,Y\in\Gamma(\mathcal{D})$, then $QY\in\Gamma(\mathcal{D})$, from Proposition \ref{p1} (d). Hence the relation (\ref{w6}) yields 
\begin{eqnarray*}
		g((\nabla_{\xi}\varphi)X,\varphi Y)=g(\xi,\nabla_XQY)-g(X,\varphi Y).
\end{eqnarray*}
A similar relation holds good by interchanging $X$ with $Y$ and if we sum these two relations, we obtain:
\begin{eqnarray*}
		&&g((\nabla_{\xi}\varphi)X,\varphi Y)+g((\nabla_{\xi}\varphi)Y,\varphi X)=g(\xi,\nabla_XQY+\nabla_YQX)-2g(X,\varphi Y)\\
		&&=g(\xi,\nabla_XY+\nabla_YX)-2g(X,\varphi Y)=(\mathcal{L}_{\xi}g)(X,Y)-2g(X,\varphi Y),
\end{eqnarray*}	
where we have used the relation	(\ref{w7}), (\ref{eq101}) and (\ref{eq1a}). On the other hand, by using (\ref{w9}), we compute the left hand side of the last relation, in which we take into account the above assertions (i) and (ii), as well as the condition (\ref{w7}), to obtain 
\begin{eqnarray*}
	g((\nabla_{\xi}\varphi)X,\varphi Y)+g((\nabla_{\xi}\varphi)Y,\varphi X)&=& 0,
\end{eqnarray*}
which gives us (\ref{w5}) and complete the proof.
\end{proof}

\begin{remark}
	In the above theorem, if the restriction of $Q$ to the distribution $\mathcal{D}$ is just the identity map up to a multiplicative function, then that function should be constant, based on the condition (\ref{w7}).
\end{remark}

\begin{remark}
 As in Remark 2 of \cite{Rov56}, the condition (\ref{w7}) can be relaxed.
\end{remark}\medskip

\noindent \textbf{{Acknowledgements}}\medskip

\noindent (i) This work is supported by Mersin University Scientific Research Projects Coordination Unit. Project Number: 2025-2-AP2-5458.\smallskip

\noindent (ii) Both authors are deeply indebted to professor Rovenski for useful suggestions that led to the improvement of our manuscript.

\end{document}